\documentclass{article}
\usepackage[utf8]{inputenc}
\usepackage{xcolor}

\usepackage[english]{babel}

\usepackage[utf8]{inputenc}
\usepackage[T1]{fontenc}

\usepackage{amssymb}
\usepackage{stmaryrd}
\usepackage{amsmath,amsfonts}
\usepackage[tikz]{bclogo}
\usepackage[top=4.34cm]{geometry}

\usepackage{lmodern}
\usepackage{textcomp}
\usepackage{mathtools, bm}
\usepackage{amssymb, bm}
\usepackage[unq]{unique}
\usepackage{enumerate}
\usepackage{comment}
\usepackage[breaklinks]{hyperref}

\usepackage[numbers]{natbib}  

\usepackage[breaklinks]{hyperref}
\usepackage[numbers]{natbib}

\usepackage{bbm}
\usepackage{amsthm}

\usepackage{cleveref}

\usepackage{appendix}

\newtheorem{theorem}{Theorem}[section]
\newtheorem{lemma}[theorem]{Lemma}

\newtheorem{conjecture}[theorem]{Conjecture}

\newtheorem{Question}[theorem]{Question}

\newtheorem{Strategy}[theorem]{Strategy}

\theoremstyle{definition}

\newcommand{\wS}{\hat{w}_i^S}
\newcommand{\q}{q}
\newcommand{\N}{N}
\newcommand{\R}{r}
\newcommand{\s}{l}
\newcommand{\E}{p}
\newcommand{\leaf}{u}
\newcommand{\iso}{u'}

\title{Bounds for the Competition-Independence game on trees}
\author{Jan Petr\footnote{\href{mailto:jp895@cam.ac.uk}{jp895@cam.ac.uk}, Department of Pure Mathematics and Mathematical Statistics (DPMMS), University of Cambridge, Wilberforce Road, Cambridge, CB3 0WA, United Kingdom} \and Julien Portier\footnote{\href{mailto:jp899@cam.ac.uk}{jp899@cam.ac.uk}, Department of Pure Mathematics and Mathematical Statistics (DPMMS), University of Cambridge, Wilberforce Road, Cambridge, CB3 0WA, United Kingdom} }
\date{}

\begin{document}

\maketitle

\begin{abstract}
In this paper we prove that Sweller has a strategy so that the Sweller-Start Competition-Independence game lasts at least $(5n+3)/13$ moves for every tree. Moreover, we show that there exist arbitrarily large trees such that the Sweller-Start Competition-Independence game lasts at most $(5n+26)/12$ moves, disproving a conjecture by Henning.
\end{abstract}

\section{Introduction}

\emph{Competition-independence games} were introduced by Phillips and Slater in 2001 in \cite{phillips2001introduction}, \cite{phillips2002graph}. With the notation of Goddard and Henning \cite{goddard2018competition}, the competition-independence game on a graph $G$ is played by two players, Diminisher (D) and Sweller (S). They take turns in constructing a maximal independent set $M$ of $G$. More precisely, the players alternate turns in which they choose a vertex that is not adjacent to any of the vertices already chosen by any of the 2 players until there is no such vertex. Upon completion of the game, the resulting set of chosen vertices is indeed a maximal independent set of $G$, and in particular a dominating set of $G$. The goal of Diminisher is to make the final set $M$ as small as possible and for Sweller to make the final set $M$ as large as possible.

Even though the competition-independence games had been introduced before 2010, when the domination games were first studied, a competition-independence game is a member of the family of the domination games. Those have received much attention, see e.g. \cite{brevsar2010domination, henning2016domination,kinnersley2013extremal,bujtas20221,versteegen2022proof, portier2023progress,henning2017_0.8,portier2022proof,bujtas2016transversal,alon2002game}, and \cite{brevsar2021domination} for a general survey.

For a graph G, let $I_s(G)$ denote the length of the competition-independence game if Sweller moves first and both players play optimally and let $I_d(G)$ denote the length of the competition-independence game if Diminisher moves first and both players play optimally. These numbers are called the competition-independence numbers. Phillips and Slater noticed \cite{phillips2002graph} that for a path on $n$ vertices we have $I_s(P_n)=\frac{3}{7}n+\theta(1)$ and wondered if this was the worst-case for Sweller among all trees. This prompted Henning \cite{henning2018my} to pose the following conjecture.

\begin{conjecture}\label{HenningConj}
Let $T$ be a tree of order $n$. Then $I_s(T) \geq \frac{3}{7}n$.
\end{conjecture}

In the first part of this paper, we disprove \Cref{HenningConj} by showing the following result.

\begin{theorem}\label{Example512}
There exist trees $T$ of size arbitrarily large satisfying $I_s(T) \leq \frac{5n+26}{12}$.
\end{theorem}

Goddard and Henning \cite{goddard2018competition} proved that, for any tree $T$ on $n$ vertices and maximum degree at most 3, we have $I_s(T)\geq 3n/8$. In the second part of our paper, we prove that a larger lower bound holds for every tree.

\begin{theorem}\label{linearLB}
Let $T$ be a tree of order $n$. Then $I_s(T) \geq \frac{5n+3}{13}$.
\end{theorem}

\section{The upper bound}

This section is devoted to proving the following result, which indeed immediately implies \Cref{Example512}.

\begin{theorem}
For every $k \in \mathbb{N}$, there exists a tree $T_k$ on $n=2 \cdot (3k+1)$ vertices such that $I_s(T_k) \leq \frac{5}{12}n+\frac{13}{6}$.
\end{theorem}

\begin{proof}
Let $S_k$ be the graph formed by $k$ paths on $3$ vertices, each connected by one of its endpoints to a central vertex. The tree $T_k$ is formed by joining two copies of $S_k$ by their central vertices.

\begin{figure}[htbp]\centering
    			\includegraphics[height=4cm]{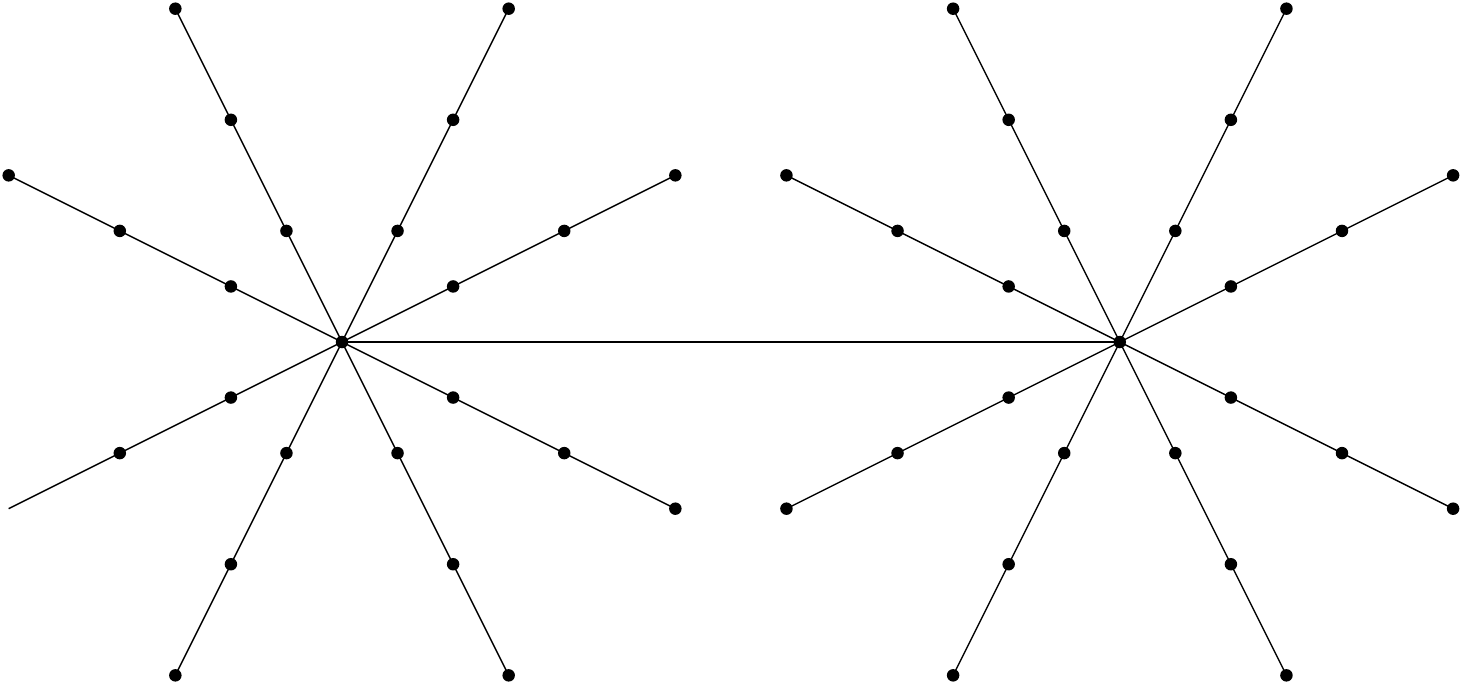}
    			\caption{The tree $T_k$.}
\label{Comp_Ind_G}
	\end{figure}

We describe here a strategy for Diminisher such that the Sweller-Start competition-independence game on $T_k$ finishes after at most $3+5k/2$ moves. If Sweller plays for its first move one of the central vertices, then Diminisher plays any legal move for its first move. Otherwise, one of the central vertices is a legal move, which Diminisher plays for its first move. Let $S^1_k$ be the copy of $S_k$ where the central vertex has been played and $S^2_k$ be the copy of $S_k$ where the central vertex has not been played. For the next moves, we instruct Diminisher to play the middle vertex on each of the paths of $S^2_k$ where no player has played yet, thus making sure only one move is played on these paths throughout the whole game. It is clear that Diminisher is able to play at least $\lfloor\frac{k-1}{2}\rfloor$ of those moves. On the remaining paths in $S^2_k$ at most two moves will be played throughout the game, and on each path of $S^1_k$ exactly one move will be played throughout the game.

In total, we have $I_s(T_k)\leq 2+\lfloor\frac{k-1}{2}\rfloor+2\cdot \lceil\frac{k+1}{2}\rceil+k \leq3+\frac{5}{2}k =\frac{5}{12}n+\frac{13}{6}$.
\end{proof}

\section{The lower bound}

We actually prove a slightly stronger result, which obviously implies \Cref{linearLB}.

\begin{theorem}\label{linearLBForest}
Let $F$ be a forest of order $n$ with $C$ connected components. Then $I_s(F) \geq \frac{5n+3C}{13}$.
\end{theorem}

The rest of this section is devoted to proving \Cref{linearLBForest}.

\subsection{The strategy and notation}

The competition-independence game can be rephrased as follows. Starting with a graph $G$, the players take turns, alternatively removing vertices together with their neighbourhoods. The game ends when the resulting graph is empty. Diminisher's goal is to make the game as short as possible, whereas Sweller's goal is to make it last as long as possible.

Considering global properties of graphs, a candidate for a heuristically good move for Sweller

\begin{itemize}
    \item removes only few vertices, and also, with the aim to make future moves remove fewer vertices,
    \item removes many edges (so that neighbourhoods get smaller), and
    \item creates many isolated vertices (as they will need to be played at some point, not removing any more vertices).
\end{itemize}

The following paragraphs are devoted to introducing Sweller's strategy for Sweller-start competition game on a forest $F$ on $C$ components, together with notation we shall use. Let $K(G)$ stand for the number of isolated vertices in a graph $G$. With the convention of $F_D^{0}=F$, let $F_i^D$ be the forest after Diminisher's $i$-th move. For any $u \in V(F_{i-1}^D)$, we set:

\begin{itemize}
    \item $v_i^S(u)=|N[u]|$, that is, the number of vertices that get removed if Sweller chooses $u$ for its move,
    \item $e_i^S(u)=|E(F_{i-1}^D)|-|E(F_{i-1}^D-N[u])|$, that is, the number of edges that get removed if Sweller chooses $u$ for its move, and
    \item $k_i^S(u)=K(F_{i-1}^D-N[u])-K(F_{i-1}^D)$, that is, the difference between the number of isolated vertices in the resulting forest if Sweller chooses $u$ and in $F_{i-1}^D$.
\end{itemize}

We instruct Sweller to follow a simple greedy strategy:

\begin{Strategy}
Sweller picks any vertex $w$ that minimizes the function $m_i^S: V(F_{i-1}^D) \rightarrow \mathbb{R}$, with $m_i^S(u)=(1-\beta)k_i^S(u) + v_i^S(u) - \alpha e_i^S(u)$, where $\alpha=3/8$ and $\beta=13/8$.
\end{Strategy}

Suppose Sweller follows this greedy strategy and consider any sequence of Diminisher's replies. Let $\N$ be the number of moves the game lasts. We will ultimately show $\N\geq(5n+3C)/13$. Note that if $\N$ is even, Diminisher plays the last move, whereas if $\N$ is odd, the last move is played by Sweller. We set $\R=\lfloor \frac{\N}{2} \rfloor$, the number of moves played by Diminisher.

Let $w_i^S$ be the vertex Sweller chooses for its $i$-th move and let $F_i^S$ be the forest after Sweller's $i$-th move, that is, $F_{i}^S=F_{i-1}^D-N[w_i^S]$. Furthermore, let $w_i^D$ be the $i$-th move of Diminisher, then $F_{i}^D=F_i^S-N[w_i^D]$. Analogously to $v_i^S, e_i^S, k_i^S$ and $m_i^S$, for any vertex $u\in V(F_i^S)$ we define:

\begin{itemize}
    \item $v_i^D(u)=|N[u]|$,
    \item $e_i^D(u)=|E(F_i^S)|-|E(F_i^S-N[u])|$,
    \item $k_i^D(u)=K(F_{i}^S-N[u])-K(F_{i}^S)$,
\end{itemize}
and $m_i^D: V(F_{i}^S) \rightarrow \mathbb{R}$ as $m_i^D(u)=(1-\beta)k_i^D(u) + v_i^D(u) - \alpha e_i^D(u)$.

It will be convenient for us to consider changes to the forest after both players make their move. In accordance with this, for all $i\leq \R$ we set:

\begin{itemize}
    \item $v_i=|V(F_D^{i-1})|-|V(F_D^{i})|$, that is, the number of vertices that get removed by both players during their $i$-th moves,
    \item $e_i=|E(F_D^{i-1})|-|E(F_D^{i})|$, that is, the number of edges that get removed by both players during their $i$-th moves, and
    \item $k_i=|K(F_D^{i})|-|K(F_D^{i-1})|$, that is, the increment in number of isolated vertices over the $i$-th moves of both players.
\end{itemize}
Finally, we set $m_i=(1-\beta)k_i + v_i - \alpha e_i=m_i^S(w_i^S)+m_i^D(w_i^D)$.

\subsection{Key lemma}

This subsection is devoted to proving the following lemma, which gives a bound on all $m_i$. The proof of \Cref{linearLBForest} straightforwardly follows, as we show in the next subsection.

Recall that $\R=\lfloor \frac{N}{2} \rfloor$ is the number of moves played by Diminisher.

\begin{lemma}\label{StrategyS}
For every $i \leq \R$, we have

$$2\beta \geq (1-\beta)k_i+v_i-\alpha e_i.$$

Moreover, if $\N$ is odd, we have

$$\beta \geq (1-\beta)k_{\R+1}^S(w_{\R+1}^S)+v_{\R+1}^S(w_{\R+1}^S)-\alpha e_{\R+1}^S(w_{\R+1}^S).$$

\end{lemma}

To prove \Cref{StrategyS}, we shall show that given Diminisher's $i$-th move $w_i^D$ Sweller could have picked a vertex $\wS$ "near" $w_i^D$ for its $i$-th move so that $2 \beta \geq m_i^S(\wS)+m_i^D(w_i^D)$. As $m_i^S(\wS) \geq m_i^S(w_i^S)$, the lemma follows.

We shall make a distinction of cases based on $v_i^D(w_i^D)$ and $e_i^D(w_i^D)$. Let $y_1,\ldots,y_{\s}$ be the neighbours of $w_i^D$ in $F_i^S$, so that $v_i^D(w_i^D)=1+\s$. Let $\E$ stand for the number of edges in $F_i^S-w_i^D$ with at least one endpoint among $\{y_1,\ldots,y_{\s}\}$, then $e_i^D(w_i^D)=\s+\E$.

\begin{lemma}\label{keylemma}
Let $i \leq r$ and suppose $v_i^D(w_i^D)=1+\s$ and $e_i^D(w_i^D)=\s+\E$ for some $\s \geq 1$. Then there is $\wS \in V(F_{i-1}^D)$ such that
\begin{itemize}
\item either $v_i^S(\wS)=\q+2$, $e_i^S(\wS) \geq \q+\s$ and $k_i^S(\wS) \geq \s-\E-2$ for some $\q \leq \lfloor \frac{\E}{\s} \rfloor$,
\item or $v_i^S(\wS)=\q+3$, $e_i^S(\wS) \geq \q+\s+2$ and $k_i^S(\wS) \geq \s-\E-1$ for some $\q \leq \lfloor \frac{\E}{\s} \rfloor$.
\end{itemize}
\end{lemma}

\begin{proof}
Let $y_1, \dots, y_{\s}$ be the neighbours of $w_i^D$ in $F_i^S$. Then by pigeonhole principle and the fact that $F_i^S$ is acyclic, there exists a vertex $y_j$ that is adjacent in $F_i^S$ to exactly $1+\q$ of the edges removed during Diminisher's $i$-th move for some $\q \leq \lfloor \frac{\E}{\s} \rfloor$. Let $y_jw_i^D$, $y_ja_1$, $\dots$, $y_ja_{\q}$ be those edges. First, note that $y_j$ has at most $2+\q$ neighbours in $F_{i-1}^D$, as otherwise $w_i^S$ would have at least $2$ common neighbours with $y_j$ in $F_{i-1}^D$, which is impossible since $F_{i-1}^D$ is acyclic.

If $y_j$ has exactly $1+\q$ neighbours in $F_{i-1}^D$, then $\wS=y_j$ was a possible Sweller's $i$-th move removing $\q+2$ vertices $y_j$, $w_i^D$, $a_1$, $\dots$, $a_{\q}$ and at least the $\q+\s$ edges $y_ja_1$, $\dots$, $y_ja_{\q}$, $y_1w_i^D$, $\dots$, $y_{\s}w_i^D$ and making at least $\s-\E-2$ vertices isolated, as at least $\s-\E-2$ of the $y_l$'s with $l \neq j$ each have $w_i^D$ as their only neighbour in $F_{i-1}^D$, leading to the first outcome of the lemma.

If $y_j$ has exactly $2+\q$ neighbours in $F_{i-1}^D$, then $w_i^S$ has exactly one common neighbour with $y_j$ in $F_{i-1}^D$, say $x$. Then $\wS=y_j$ was a possible Sweller's $i$-th move removing the $\q+3$ vertices $w_i^S$, $y_j$, $x$, $a_1$, $\dots$, $a_q$ and at least the $\q+\s+2$ edges $y_jx$, $w_i^Sx$, $y_ja_1$, $\dots$, $y_ja_{\q}$, $y_1w_i^D$, $\dots$, $y_{\s}w_i^D$ and making at least $\s-\E-1$ vertices isolated, as at least $\s-\E-1$ of the $y_l$'s with $l \neq j$ each have $w_i^D$ as their only neighbour in $F_{i-1}^D$, leading to the second outcome of the lemma.
\end{proof}

\begin{lemma}\label{copymove}
Let $i \leq \R$ and suppose $v_i^D(w_i^D)=1+\s$ and $e_i^D(w_i^D)=\s+\E$  for some $\s \geq 1$. Then there is $\wS \in V(F_{i-1}^D)$ such that
\begin{itemize}
\item either $v_i^S(\wS)=\s+1$ and $e_i^S(\wS) \geq \E+\s$,
\item or $v_i^S(\wS)=\s+2$ and $e_i^S(\wS) \geq \E+\s+2$.
\end{itemize}
\end{lemma}

\begin{proof}
Set $\wS=w_i^D$. Then, in the same spirit as the proof of the previous lemma, distinguishing upon $w_i^D$ having $1+\s$ or $2+\s$ neighbours in $F_{i-1}^D$, we get each of the two outcomes claimed in the lemma.
\end{proof}

The next lemma is a refinement of the two previous lemmas for small values of $\s$ and $\E$.

\begin{lemma}\label{SpecialCases}
Suppose $v_i^D(w_i^D)=1+\s$ and $e_i^D(w_i^D)=\s+\E$.

\begin{enumerate}
\item If $\s=2$ and $\E=0$, then there is $\wS \in V(F_{i-1}^D)$ such that
\begin{itemize}
\item either $v_i^S(\wS)=2$, $e_i^S(\wS) \geq 2$ and $k_i^S(\wS) \geq 1$,
\item or $v_i^S(\wS)=3$, $e_i^S(\wS) \geq 4$ and $k_i^S(\wS) \geq 1$.
\end{itemize}
\item If $\s=3$ and $\E=0$, then there is $\wS \in V(F_{i-1}^D)$ such that $v_i^S(\wS)=2$, $e_i^S(\wS) \geq 3$ and $k_i^S(\wS) \geq 1$. 
\item If $\s=3$ and $\E=1$, then there is $\wS \in V(F_{i-1}^D)$ such that \begin{itemize}
\item either $v_i^S(\wS)=2$, $e_i^S(\wS) \geq 3$ and $k_i^S(\wS) \geq 1$,
\item or $v_i^S(\wS)=3$, $e_i^S(\wS) \geq 5$ and $k_i^S(\wS) \geq 1$. 
\end{itemize}
\item If $\s=2$ and $\E=1$ or $3 \leq \s=\E \leq 4$, then there is $\wS \in V(F_{i-1}^D)$ such that $v_i^S(\wS)=2$ and $e_i^S(\wS) \geq 2$.
\item If $\s=1$ and $\E=0$, then there is $\wS \in V(F_{i-1}^D)$ such that $v_i^S(\wS)=2$ and $e_i^S(\wS) \geq 1$.
\item If $\s=0$ and $\E=0$, then there is $\wS \in V(F_{i-1}^D)$ such that 
\begin{itemize}
\item either $v_i^S(\wS)=2$ and $e_i^S(\wS) \geq 1$,
\item or $v_i^S(\wS)=1$ and $k_i^S(\wS)=-1$.
\end{itemize}
\end{enumerate}
\end{lemma}

\begin{proof}
In each case, let $y_1, \dots, y_{\s}$ be the neighbours of $w_i^D$ in $F_i^{S}$.
\begin{enumerate}
    \item In this case, observe that in the forest $F_{i-1}^D$, we must be in one of those two subcases (up to a permutation of the vertices $y_1$ and $y_2$): either both $y_1$ and $y_2$ have exactly one neighbour or $y_1$ has exactly one neighbour $w_i^D$ and $y_2$ has exactly two neighbours $w_i^D$ and $x$, such that $x$ is a neighbour of $w_i^S$.
    
    In the first subcase, $\wS=y_2$ was a possible Sweller's $i$-th move removing $2$ vertices $y_2$ and $w_i^D$, removing at least the $2$ edges $y_2w_i^D$ and $y_1w_i^D$ and making at least the vertex $y_1$ isolated.
    
    In the second subcase, $\wS=y_2$ was a possible Sweller's $i$-th move removing $3$ vertices $x$, $y_2$ and $w_i^D$, removing at least $4$ edges $xw_i^S$, $xy_2$, $y_2w_i^D$ and $y_1w_i^D$ and making at least the vertex $y_1$ isolated.
    
    \item In this case, observe that in $F_{i-1}^D$, at most one of $y_1$, $y_2$ and $y_3$ has another neighbour, say $y_3$. Then $\wS=y_1$ was a possible Sweller's $i$-th move removing $2$ vertices $y_1$ and $w_i^D$, removing at least $3$ edges $y_1w_i^D$, $y_2w_i^D$ and $y_3w_i^D$ and making at least the vertex $y_2$ isolated.
    
    \item In this case, observe that in $F_{i-1}^D$, at least one of $y_1$, $y_2$ and $y_3$ has exactly one neighbour, say $y_1$. Then the vertex $y_1$ was a possible Sweller's $i$-th move removing $2$ vertices $y_1$ and $w_i^D$ and removing at least $3$ edges $w_i^Dy_1$, $w_i^Dy_2$ and $w_i^Dy_3$. If $k_i^S(y_1) \geq 1$, set $\wS=y_1$.
    
    Now, suppose $k_i^S(y_1)=0$. Then each of $y_2, y_3$ has exactly one neighbour other than $w_i^D$ in $F_{i-1}^D$, and consequently exactly one of $y_2, y_3$, say $y_3$, has a common neighbour $x$ with $w_i^S$ in $F_{i-1}^D$. Then $\wS=y_3$ was a possible Sweller's $i$-th move, removing $3$ vertices $w_i^D, y_3, x$, removing at least $5$ edges $w_i^Dy_1, w_i^Dy_2, w_i^Dy_3, y_3x$ and $xw_i^S$ and making at least the vertex $y_1$ isolated.
    
    \item In this case, observe that there exists a leaf $\leaf$ in $F_{i-1}^D$ such that its neighbour has degree at least $2$. Then $\wS=\leaf$ was a possible Sweller's $i$-th move removing $2$ vertices and at least $2$ edges.
    
    \item In this case, let $\leaf$ be a leaf of $F_{i-1}^D$. Then $\wS=\leaf$ was a possible Sweller's $i$-th move removing $2$ vertices and at least $1$ edge.
    
    \item In this case, $F_{i-1}^D$ contains either a leaf $\leaf$ or an isolated vertex $\iso$.  Then it is straightforward to check that $\wS=\leaf$ leads to the first outcome and $\wS=\iso$ leads to the second.
\end{enumerate}
\end{proof}

All that remains to prove the key lemma is a matter of casework and simple calculations, the details of which can be found in \Cref{appendix}.

\subsection{Bounding the number of moves}

We now prove \Cref{linearLB} using \Cref{StrategyS}. Recall that $\N$ is the total number of moves played by both players. Since the number of vertices removed throughout the game is exactly the number $n$ of vertices in $F$, we have
\begin{align}\label{eq:EqVertices}
     n=\sum_{i=1}^{\lceil N/2 \rceil} v_i^S(w_i^S)+\sum_{i=1}^{\lfloor N/2 \rfloor} v_i^D(w_i^D) = \sum_{i=1}^{r} v_i + \mathbbm{1}_{\{2\nmid N\}}\cdot v_{\R+1}^S(w_{\R+1}^S).
\end{align}

Similarly, the number of edges removed throughout the game is exactly the number of edges in $F$, hence we have
\begin{align}\label{eq:EqEdges}
    n-C=\sum_{i=1}^{\lceil N/2 \rceil} e_i^S(w_i^S)+\sum_{i=1}^{\lfloor N/2 \rfloor} e_i^D(w_i^D) = \sum_{i=1}^{\R} e_i + \mathbbm{1}_{\{2\nmid N\}}\cdot e_{\R+1}^S(w_{\R+1}^S).
\end{align}

Moreover, as $k_i$ is the increment of the number of isolated vertices, and that there are no isolated vertices once the game ends, we have
\begin{align}\label{eq:IneqIso}
    0\geq\sum_{i=1}^{\lceil N/2 \rceil} k_i^S(w_i^S)+\sum_{i=1}^{\lfloor N/2 \rfloor} k_i^D(w_i^D) = \sum_{i=1}^{\R} k_i + \mathbbm{1}_{\{2\nmid N\}}\cdot k_{\R+1}^S(w_{\R+1}^S).
\end{align}

Using \Cref{StrategyS} and then plugging \eqref{eq:EqVertices}, \eqref{eq:EqEdges} and \eqref{eq:IneqIso} gives

\[\begin{aligned}
N &= \sum_{i=1}^{\lfloor N/2 \rfloor} 2 + \mathbbm{1}_{\{2\nmid N\}} \\
&\geq \frac{1}{\beta}(\sum_{i=1}^{\R} (1-\beta)k_i+v_i-\alpha e_i) + \frac{1}{\beta}\mathbbm{1}_{\{2\nmid N\}}\cdot ((1-\beta)k_{\R+1}^S(w_i^S)+v_{\R+1}^S(w_i^S)-\alpha e_{\R+1}^S(w_i^S))\\
&\geq \frac{1-\alpha}{\beta}n+\frac{\alpha C}{\beta},
\end{aligned}\]
hence, $N \geq \frac{5n+3C}{13}$ as wanted.\qed

\section{Concluding remarks and further directions}

The following question remains unsolved.

\begin{Question}\label{QuestionOptimalConstant}
What is the value of $c=\inf \{ \dfrac{I_s(T)}{|V(T)|}: \text{T is a tree}\}$?
\end{Question}

In this paper, we showed that we have $5/13 \leq c \leq 5/12$. The values of the constants $\alpha$ and $\beta$ have been optimised but the authors believe that with the same method and a more careful case analysis, it is likely that one can improve \Cref{linearLB} to a lower bound $c'n$ for some $c' > 5/13$. However, it does not seem that the technique developped here could lead to anything better than a lower bound of $2n/5$ for the following reason. With the notations from the previous section, if Diminisher's first move has $\s=\E=n/5$, then the best we can ensure with the reasoning we had above is that Sweller's previous move removed constantly many vertices and roughly $n/5$ edges from the graph. Then our technique cannot rule out that the resulting graph is now roughly a disjoint union of $2n/5$ $K_2$'s, and consequently that the game will last roughly $2n/5$ more moves. Hence, if it turns out that $c>2/5$ it seems necessary that a proof of such a lower bound will require analysis of the general structure of the graph and not just of the local structure around Diminisher's moves.

Note that one can ask \Cref{QuestionOptimalConstant} for forests instead of trees as it is not clear whether the two problems have the same answer.

The authors would also like to point out that the following conjecture posed by Goddard and Henning \cite{goddard2018competition} is yet to be settled.

\begin{conjecture}
Let $T$ be a tree of order $n \geq 2$, then $I_d(T) \leq \frac{3}{4}n$.
\end{conjecture}

\section*{Acknowledgement}

The authors would like to thank their PhD supervisor Béla Bollobás for his valuable comments.

\bibliographystyle{abbrvnat}  
\renewcommand{\bibname}{Bibliography}
\bibliography{bibliography}

\appendix
\section{Proof of \Cref{StrategyS}}
\label{appendix}

This result will mainly follow from \Cref{keylemma} and \Cref{copymove}, except for a few small cases that we will treat using \Cref{SpecialCases}.

First, we show $2\beta \geq (1-\beta)k_i+v_i-\alpha e_i$ for all $i \leq \R$. Suppose $v_i^D(w_i^D)=1+\s$ and $e_i^D(w_i^D)=\s+\E$. We split all possible pairs of $(\s,\E)$ as follows (see \Cref{cases} for illustration):

\begin{figure}[htbp]\centering
    			\includegraphics[height=8cm]{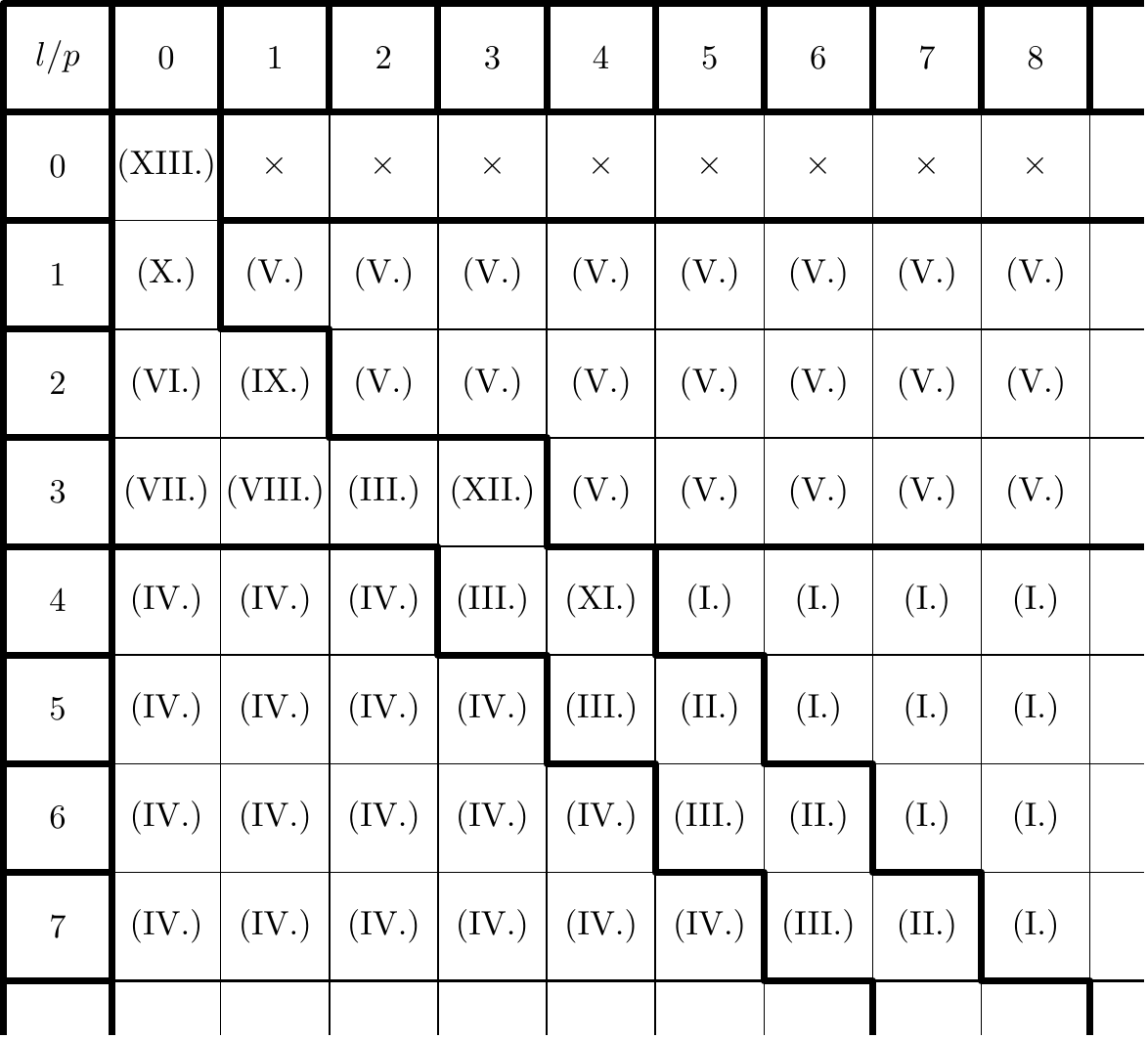}
    			\caption{An illustration of how cases (I.)--(XIII.) cover all pairs $(\s,\E)$.}
\label{cases}
\end{figure}

\begin{enumerate}[(I.)]
    \item First, suppose $\E \geq \s+1$ and $\s \geq 4$. By \Cref{keylemma}, we have:
\begin{itemize}
    \item either $(1-\beta)k_i+v_i-\alpha e_i \leq (\q+2+\s+1)-\frac{3}{8}(\q+\s+\s+\E),$
    \item or $(1-\beta)k_i+v_i-\alpha e_i \leq (\q+3+\s+1)-\frac{3}{8}(\q+\s+2+\s+\E).$
\end{itemize}

Note that $(\q+2+\s+1)-\frac{3}{8}(\q+\s+\s+\E) \leq (\q+3+\s+1)-\frac{3}{8}(\q+\s+2+\s+\E)$, so it suffices to prove that $(\q+3+\s+1)-\frac{3}{8}(\q+\s+2+\s+\E) \leq 2\beta=\frac{13}{4}$, which is indeed true:

\[\begin{aligned}
\frac{13}{4}-((\q+3+\s+1)-\frac{3}{8}(\q+\s+2+\s+\E)) &=\frac{3}{8}\E-\frac{1}{4}\s-\frac{5}{8}\q \\
&\geq \frac{3}{8}\E-\frac{1}{4}\s-\frac{5\E}{8\s}. \\
\end{aligned}\]

Using the assumptions $\E \geq \s+1$ and $\s \geq 4$, we get:

\[\begin{aligned}
\frac{3}{8}\E-\frac{1}{4}\s-\frac{5\E}{8\s} &\geq (\s+1)\frac{3\s-5}{8\s} - \frac{1}{4}\s \\
&= \frac{\s^2-2\s-5}{8\s} \\
&\geq 0. \\
\end{aligned}\]

\item Suppose $\E = \s$ and $\s \geq 5$. This is very similar to the previous case; it suffices to prove that $(\q+3+\s+1)-\frac{3}{8}(\q+\s+2+\s+\E) \leq 2\beta=\frac{13}{4}$, which is indeed true:

\[\begin{aligned}
\frac{13}{4}-((\q+3+\s+1)-\frac{3}{8}(\q+\s+2+\s+\E)) &=\frac{3}{8}\E-\frac{1}{4}\s-\frac{5}{8}\q \\
&\geq \frac{3}{8}\E-\frac{1}{4}\s-\frac{5}{8} \\
&=\frac{\s-5}{8} \\
&\geq 0. \\
\end{aligned}\]

\item Suppose $\E=\s-1$ and $\s \geq 3$. Note that $\q=0$, and that as in the previous case, it suffices to prove that $(\q+3+\s+1)-\frac{3}{8}(\q+\s+2+\s+\E) \leq 2\beta=\frac{13}{4}$, which is indeed true:

\[\begin{aligned}
\frac{13}{4}-((\q+3+\s+1)-\frac{3}{8}(\q+\s+2+\s+\E)) &=\frac{3}{8}\E-\frac{1}{4}\s \\
&=\frac{\s-3}{8} \\
&\geq 0. \\
\end{aligned}\]

\item Suppose $\E \leq \s-2$ and $\s \geq 4$. Note that $\q=0$ and by \Cref{keylemma}, we have:
\begin{itemize}
    \item either $(1-\beta)k_i+v_i-\alpha e_i \leq -\frac{5}{8}(\s-2-\E)+(2+\s+1)-\frac{3}{8}(\s+\s+\E),$
    \item or $(1-\beta)k_i+v_i-\alpha e_i \leq -\frac{5}{8}(\s-1-\E)+(3+\s+1)-\frac{3}{8}(\s+2+\s+\E).$
\end{itemize}

Note that $-\frac{5}{8}(\s-2-\E)+(2+\s+1)-\frac{3}{8}(\s+\s+\E) \geq -\frac{5}{8}(\s-1-\E)+(3+\s+1)-\frac{3}{8}(\s+2+\s+\E)$ so it suffices to prove that $-\frac{5}{8}(\s-2-\E)+(2+\s+1)-\frac{3}{8}(\s+\s+\E) \leq 2\beta=\frac{13}{4}$, which is indeed true:

\[\begin{aligned}
\frac{13}{4}-(-\frac{5}{8}(\s-2-\E)+(2+\s+1)-\frac{3}{8}(\s+\s+\E)) &= -1 - \frac{1}{4}\E + \frac{3}{8}\s \\
&\geq  -1 -\frac{1}{4}(\s-2)+\frac{3}{8}\s \\
&= \frac{\s-4}{8} \\
&\geq 0. \\
\end{aligned}\]

\item Suppose $5\s-3\E \leq 4$ and $l \geq 1$, then by \Cref{copymove}, we have: \begin{itemize}
    \item either $(1-\beta)k_i+v_i-\alpha e_i \leq (2+2\s)-\frac{3}{8}(2\s+2\E),$
    \item or $(1-\beta)k_i+v_i-\alpha e_i \leq (3+2\s)-\frac{3}{8}(2\s+2\E+2).$
\end{itemize}

Note that $(3+2\s)-\frac{3}{8}(2\s+2\E+2) \geq (2+2\s)-\frac{3}{8}(2\s+2\E)$ so it suffices to prove that $(3+2\s)-\frac{3}{8}(2\s+2\E+2) \leq 2\beta=\frac{13}{4}$, which is indeed true:

\[\begin{aligned}
\frac{13}{4}-((3+2\s)-\frac{3}{8}(2\s+2\E+2)) = \frac{4-5\s+3\E}{4} \geq 0. \\
\end{aligned}\]

\item If $\s=2$ and $\E=0$, then by \Cref{SpecialCases}, we have either $(1-\beta)k_i+v_i-\alpha e_i \leq (1-\beta)+5-4\alpha= \frac{23}{8} \leq 2\beta$ or $(1-\beta)k_i+v_i-\alpha e_i \leq (1-\beta)+6-6\alpha= \frac{25}{8} \leq 2\beta$.
\item If $\s=3$ and $\E=0$, then by \Cref{SpecialCases}, we have $(1-\beta)k_i+v_i-\alpha e_i \leq (1-\beta)+6-6\alpha = \frac{25}{8} \leq 2\beta$.
\item If $\s=3$ and $\E=1$, then by \Cref{SpecialCases}, we have either $(1-\beta)k_i+v_i-\alpha e_i \leq (1-\beta)+6-7\alpha = \frac{22}{8} \leq 2\beta$ or $(1-\beta)k_i+v_i-\alpha e_i \leq (1-\beta)+7-9\alpha = 3 \leq 2\beta$.
\item If $\s=2$ and $\E=1$, then by \Cref{SpecialCases}, we have $(1-\beta)k_i+v_i-\alpha e_i \leq 5-5\alpha = \frac{25}{8} \leq 2\beta$.
\item If $\s=1$ and $\E=0$, then by \Cref{SpecialCases}, we have $(1-\beta)k_i+v_i-\alpha e_i \leq 4-2\alpha = \frac{13}{4} = 2\beta$.
\item If $\s=4$ and $\E=4$, then by \Cref{SpecialCases}, we have $(1-\beta)k_i+v_i-\alpha e_i \leq 7-10\alpha = \frac{13}{4} = 2\beta$.
\item If $\s=3$ and $\E=3$, then by \Cref{SpecialCases}, we have $(1-\beta)k_i+v_i-\alpha e_i \leq 6-8\alpha = 3 \leq 2\beta$.
\item If $\s=0$ and $\E=0$, then by \Cref{SpecialCases}, we have either $(1-\beta)k_i+v_i-\alpha e_i \leq 3-\alpha = \frac{21}{8} \leq 2\beta$, or $(1-\beta)k_i+v_i-\alpha e_i \leq -(1-\beta)+2 = \frac{21}{8} \leq 2\beta$.
\end{enumerate}

This concludes the proof of the first claim.

For the second claim, suppose $\N$ is odd. Note that in $F_{\R+1}^D$ there is a leaf $\leaf$ or an isolated vertex $\iso$. Set $\hat{w}_{\R+1}^S$ to be either $\leaf$ or $\iso$, whichever is present. The first case leads to removing $2$ vertices and at least $1$ edge, the second leads to removing $1$ vertex, $0$ edges and $1$ isolated vertex. Hence, $\beta \geq m_{\R+1}^S(\hat{w}_{\R+1}^S) \geq m_{\R+1}^S(w_{\R+1}^S)$, as wanted. \qed

\end{document}